\documentclass[a4paper,12pt,leqno]{article}

\usepackage{amsmath}                                                               %
\usepackage{amssymb}                                                               %
\usepackage{amsthm}                                                                %
\usepackage{authblk}                                                               %
\usepackage{enumerate}                                                             %
\usepackage{float}                                                                 %
\usepackage{indentfirst}                                                           %
\usepackage{interval}                                                              %
\usepackage{mathtools}                                                             %
\usepackage{tensor}                                                                %
\usepackage{xcolor}                                                                %
\intervalconfig{soft open fences}                                                  %
\newtheorem{theorem}{Theorem}[section]                                             %
\newtheorem{corollary}{Corollary}[theorem]                                         %
\newtheorem{lemma}[theorem]{Lemma}                                                 %
\newtheorem{proposition}{Proposition}                                              %
\newtheorem{definition}{Definition}                                                %
\newtheorem{example}{Example}                                                      %
\newtheorem{remark}{Remark}                                                        %
\newcommand{\naturals}{\mathbb{N}}                                                 %
\newcommand{\reals}{\mathbb{R}}                                                    %
\newcommand{\euclideanspace}[1]{\mathbb{R}^{#1}}                                   %
\newcommand{\hyperbolicspace}[1]{\mathbb{H}^{#1}}                                  %
\newcommand{\roundsphere}[1]{\mathbb{S}^{#1}}                                      %
\DeclarePairedDelimiter\abs{\lvert}{\rvert}                                        %
\newcommand{\norm}[1]{\lVert{#1}\rVert}                                            %
\DeclarePairedDelimiter\innerproduct{\langle}{\rangle}                             %
\NewDocumentCommand{\smoothfunctions}{d ()}{                                       %
  \IfValueTF{#1}{                                                                  %
    C^{\infty}{\left(#1\right)}                                                    %
  }{                                                                               %
    C^{\infty}                                                                     %
  }                                                                                %
}                                                                                  %
\NewDocumentCommand{\smoothvectorfields}{d ()}{                                    %
  \IfValueTF{#1}{                                                                  %
    \mathfrak{X}^{\infty}{\left(#1\right)}                                         %
  }{                                                                               %
    \mathfrak{X}^{\infty}                                                          %
  }                                                                                %
}                                                                                  %
\NewDocumentCommand{\diffeomorphisms}{d ()}{                                       %
  \IfValueTF{#1}{                                                                  %
    \mathcal{D}\left(#1\right)                                                     %
  }{                                                                               %
    \mathcal{D}                                                                    %
  }                                                                                %
}                                                                                  %
\NewDocumentCommand{\riemannianmetrics}{d ()}{                                     %
  \IfValueTF{#1}{                                                                  %
    \mathcal{M}\left(#1\right)                                                     %
  }{                                                                               %
    \mathcal{M}                                                                    %
  }                                                                                %
}                                                                                  %
\NewDocumentCommand{\gradientfield}{d <> m}{                                       %
  \IfValueTF{#1}{                                                                  %
    \textup{grad}_{#1}\left(#2\right)                                              %
  }{                                                                               %
    \textup{grad}\left(#2\right)                                                   %
  }                                                                                %
}                                                                                  %
\NewDocumentCommand{\hessian}{d <> m d [] d ()}{                                   %
  \IfValueTF{#1}{                                                                  %
    \IfValueTF{#3}{                                                                %
      \IfValueTF{#4}{                                                              %
        \textup{Hess}_{#1}\left(#2\right)\left(#3,#4\right)                        %
      }{                                                                           %
        \textup{Hess}_{#1}\left(#2\right)\left(#3,\cdot\right)                     %
      }                                                                            %
    }{                                                                             %
      \IfValueTF{#4}{                                                              %
        \textup{Hess}_{#1}\left(#2\right)\left(\cdot,#4\right)                     %
      }{                                                                           %
        \textup{Hess}_{#1}\left(#2\right)                                          %
      }                                                                            %
    }                                                                              %
  }{                                                                               %
    \IfValueTF{#3}{                                                                %
      \IfValueTF{#4}{                                                              %
        \textup{Hess}\left(#2\right)\left(#3,#4\right)                             %
      }{                                                                           %
        \textup{Hess}\left(#2\right)\left(#3,\cdot\right)                          %
      }                                                                            %
    }{                                                                             %
      \IfValueTF{#4}{                                                              %
        \textup{Hess}\left(#2\right)\left(\cdot,#4\right)                          %
      }{                                                                           %
        \textup{Hess}\left(#2\right)                                               %
      }                                                                            %
    }                                                                              %
  }                                                                                %
}                                                                                  %
\NewDocumentCommand{\laplacian}{d <> m}{                                           %
  \IfValueTF{#1}{                                                                  %
    \left(\Delta{#2}\right)_{#1}                                                   %
  }{                                                                               %
    \Delta{#2}                                                                     %
  }                                                                                %
}                                                                                  %
\NewDocumentCommand{\driftedlaplacian}{d <> m m}{                                  %
  \IfValueTF{#1}{                                                                  %
    \left(\Delta_{#2}{#3}\right)_{#1}                                              %
  }{                                                                               %
    \Delta_{#2}{#3}                                                                %
  }                                                                                %
}                                                                                  %
\newcommand{\liederivative}[2]{\mathcal{L}_{#2}{#1}}                               %
\NewDocumentCommand{\curvaturetensor}{d <> m m d ()}{                              %
  \IfValueTF{#1}{                                                                  %
    \IfValueTF{#4}{                                                                %
      \textup{Rm}_{#1}\left(#2,#3\right){#4}                                       %
    }{                                                                             %
      \textup{Rm}_{#1}\left(#2,#3\right)                                           %
    }                                                                              %
  }{                                                                               %
    \IfValueTF{#4}{                                                                %
      \textup{Rm}\left(#2,#3\right){#4}                                            %
    }{                                                                             %
      \textup{Rm}\left(#2,#3\right)                                                %
    }                                                                              %
  }                                                                                %
}                                                                                  %
\NewDocumentCommand{\riccitensor}{d <> d [] d ()}{                                 %
  \IfValueTF{#1}{                                                                  %
    \IfValueTF{#2}{                                                                %
      \IfValueTF{#3}{                                                              %
        \textup{Rc}_{#1}\left(#2,#3\right)                                         %
      }{                                                                           %
        \textup{Rc}_{#1}\left(#2,\cdot\right)                                      %
      }                                                                            %
    }{                                                                             %
      \IfValueTF{#3}{                                                              %
        \textup{Rc}_{#1}\left(\cdot,#3\right)                                      %
      }{                                                                           %
        \textup{Rc}_{#1}                                                           %
      }                                                                            %
    }                                                                              %
  }{                                                                               %
    \IfValueTF{#2}{                                                                %
      \IfValueTF{#3}{                                                              %
        \textup{Rc}\left(#2,#3\right)                                              %
      }{                                                                           %
        \textup{Rc}\left(#2,\cdot\right)                                           %
      }                                                                            %
    }{                                                                             %
      \IfValueTF{#3}{                                                              %
        \textup{Rc}\left(\cdot,#3\right)                                           %
      }{                                                                           %
        \textup{Rc}                                                                %
      }                                                                            %
    }                                                                              %
  }                                                                                %
}                                                                                  %
\NewDocumentCommand{\scalarcurvature}{d <>}{                                       %
  \IfValueTF{#1}{                                                                  %
    \textup{R}_{#1}                                                                %
  }{                                                                               %
    \textup{R}                                                                     %
  }                                                                                %
}                                                                                  %
\newcommand{\ambientspace}[2]{\mathbb{M}_{#1}^{#2}}                                %
\newcommand{\submanifold}{\textup{M}^{n}}                                          %
\newcommand{\tangentspace}[2]{\textup{T}_{#1}#2}                                   %
\newcommand{\normalspace}[2]{\textup{T}^{\perp}_{#1}#2}                            %
\newcommand{\normalfield}{\textup{N}}                                              %
\NewDocumentCommand{\ambientconn}{d [] d ()}{                                      %
  \IfValueTF{#1}{                                                                  %
    \IfValueTF{#2}{                                                                %
      \bar{\nabla}_{#1}{#2}                                                        %
    }{                                                                             %
      \bar{\nabla}_{#1}                                                            %
    }                                                                              %
  }{                                                                               %
    \IfValueTF{#2}{                                                                %
      \bar{\nabla}{#2}                                                             %
    }{                                                                             %
      \bar{\nabla}                                                                 %
    }                                                                              %
  }                                                                                %
}                                                                                  %
\NewDocumentCommand{\submanifoldconn}{d [] d ()}{                                  %
  \IfValueTF{#1}{                                                                  %
    \IfValueTF{#2}{                                                                %
      \nabla_{#1}{#2}                                                              %
    }{                                                                             %
      \nabla_{#1}                                                                  %
    }                                                                              %
  }{                                                                               %
    \IfValueTF{#2}{                                                                %
      \nabla{#2}                                                                   %
    }{                                                                             %
      \nabla                                                                       %
    }                                                                              %
  }                                                                                %
}                                                                                  %
\NewDocumentCommand{\secondfundamentalform}{d [] d ()}{                            %
  \IfValueTF{#1}{                                                                  %
    \IfValueTF{#2}{                                                                %
      \textup{h}\left(#1,#2\right)                                                 %
    }{                                                                             %
      \textup{h}\left(#1,\cdot\right)                                              %
    }                                                                              %
  }{                                                                               %
    \IfValueTF{#2}{                                                                %
      \textup{h}\left(\cdot,#2\right)                                              %
    }{                                                                             %
      \textup{h}                                                                   %
    }                                                                              %
  }                                                                                %
}                                                                                  %
\NewDocumentCommand{\shapeoperator}{d []}{                                         %
  \IfValueTF{#1}{                                                                  %
    \textup{A}{#1}                                                                 %
  }{                                                                               %
    \textup{A}                                                                     %
  }                                                                                %
}                                                                                  %
\NewDocumentCommand{\trace}{d ()}{                                                 %
  \IfValueTF{#1}{                                                                  %
    \textup{tr}{\left(#1\right)}                                                   %
  }{                                                                               %
    \textup{tr}                                                                    %
  }                                                                                %
}                                                                                  %
\newcommand{\meancurvature}{\textup{H}}                                            %
\newcommand{\identityoperator}{\textup{id}}                                        %
\newcommand{\tracelessshapeoperator}{\Phi}                                         %
\NewDocumentCommand{\ambientriccitensor}{d <> d [] d ()}{                          %
  \IfValueTF{#1}{                                                                  %
    \IfValueTF{#2}{                                                                %
      \IfValueTF{#3}{                                                              %
        \bar{\textup{R}}\textup{c}_{#1}\left(#2,#3\right)                          %
      }{                                                                           %
        \bar{\textup{R}}\textup{c}_{#1}\left(#2,\cdot\right)                       %
      }                                                                            %
    }{                                                                             %
      \IfValueTF{#3}{                                                              %
        \bar{\textup{R}}\textup{c}_{#1}\left(\cdot,#3\right)                       %
      }{                                                                           %
        \bar{\textup{R}}\textup{c}_{#1}                                            %
      }                                                                            %
    }                                                                              %
  }{                                                                               %
    \IfValueTF{#2}{                                                                %
      \IfValueTF{#3}{                                                              %
        \bar{\textup{R}}\textup{c}\left(#2,#3\right)                               %
      }{                                                                           %
        \bar{\textup{R}}\textup{c}\left(#2,\cdot\right)                            %
      }                                                                            %
    }{                                                                             %
      \IfValueTF{#3}{                                                              %
        \bar{\textup{R}}\textup{c}\left(\cdot,#3\right)                            %
      }{                                                                           %
        \bar{\textup{R}}\textup{c}                                                 %
      }                                                                            %
    }                                                                              %
  }                                                                                %
}                                                                                  %
\title{Classification of gradient Yamabe soliton hypersurfaces of space forms}     %
\author[,1]{Willian Isao Tokura\thanks{email: williantokura@ufgd.edu.br}}          %
\author[,2]{Marcelo Bezerra Barboza \thanks{email: bezerra@ufg.br}}                %
\affil[1]{Universidade Federal da Grande Dourados}                                 %
\affil[2]{Universidade Federal de Goiás}                                           %
\date{\today}                                                                      %
\begin{document}

\maketitle

\begin{abstract}
  In this paper we investigate gradient Yamabe solitons, either shrinking or
  steady, that can be isometrically immersed into space forms as hypersurfaces
  that admit an upper bound on the norm of their second fundamental form. Those
  solitons satisfying an additional condition, that could be constant mean
  curvature or the number of critical points of the potential function being at
  most one, are fully classified. Our argument is based on the weak Omori-Yau
  principle for the drifted Laplacian on Riemannian manifolds.
\end{abstract}

\section{Introduction}\label{sec:introduction}

Given an \(n\)-dimensional differentiable manifold \(M\), let's write
\(\diffeomorphisms\) for the group of diffeomorphisms of \(M\) and
\(\riemannianmetrics\) for the space of Riemannian metrics on \(M\). Recall the
right action of \(\diffeomorphisms\) on \(\riemannianmetrics\) by pullbacks
\[
  \riemannianmetrics\times{\diffeomorphisms}\to{\riemannianmetrics},
  \quad
  (g,\phi)\mapsto{\phi_{\ast}(g)}.
\]
The Yamabe flow~\cite{hamilton1988ricciflow} starting from a given Riemannian
manifold \((M,g)\) is a \(1\)-parameter family
\(\interval{0}{T}\ni{t}\mapsto{g^{t}}\in{\riemannianmetrics}\) of Riemannian
metrics such that
\begin{equation}\label{eq:yamabe-flow}
  \left\{
    \begin{array}{rcl}
      \dfrac{\partial{g^{t}}}{\partial{t}}
      & = &
      -\scalarcurvature^{t}\cdot{g^{t}},
      \\ [0.3cm]
      g^{0}
      & = &
      g,
    \end{array}
    \right.
\end{equation}
where \(\scalarcurvature^{t}\) is the scalar curvature of \(g^{t}\). Yamabe
solitons are self-similar solutions of~\eqref{eq:yamabe-flow}, meaning that
\begin{equation}\label{eq:yamabe-soliton-1}
  \left\{
    \begin{array}{lcl}
      g^{t}      & = & \alpha^{t}\cdot\psi^{t}_{\ast}{\left(g^{0}\right)}, \\
      \alpha^{0} & = & 1,                                                  \\
      \psi^{0}   & = & \textup{id}_{M},
    \end{array}
    \right.
\end{equation}
for certain smooth function
\(\interval{0}{T}\ni{t}\mapsto{\alpha^{t}}\in{\ointerval{0}{\infty}}\) and
\(1\)-parameter family
\(\interval{0}{T}\ni{t}\mapsto{\psi^{t}}\in{\diffeomorphisms}\) of
diffeomorphisms of \(M\).

By differentiating the first equation of~\eqref{eq:yamabe-soliton-1} with
respect to \(t\) and evaluating the expression then obtained at \(t=0\) we find
yet another way of describing a Yamabe soliton.

\begin{definition}\label{def:yamabe-soliton}
  A Yamabe soliton is a Riemannian manifold \((M^{n},g)\) together with a
  vector field \(X\in{\smoothvectorfields(M)}\) and a real number
  \(\lambda\in{\reals}\) satisfying
  \begin{equation}\label{eq:yamabe-soliton-2}
    \frac{1}{2}
    \liederivative{g}{X}
    =
    (\scalarcurvature-\lambda)g
    ,
  \end{equation}
  where \(\liederivative{g}{X}\) is the Lie derivative of \(g\) with respect to
  \(X\) and \(\scalarcurvature\) is the scalar curvature of \(g\). It's called
  shrinking if \(\lambda>0\), steady if \(\lambda=0\) or expanding if
  \(\lambda<0\). If \(\liederivative{g}{X}=0\), it's called trivial. We write
  the Yamabe soliton given by~\eqref{eq:yamabe-soliton-2} like
  \((M^{n},g,X,\lambda)\) in order to emphasize all of its components.
\end{definition}

The next definition deals with those Yamabe solitons \((M,g,X,\lambda)\) in
which the vector field \(X\) is the gradient field of some scalar function
\(f\in{\smoothfunctions(M)}\).

\begin{definition}\label{def:gradient-yamabe-soliton}
  A gradient Yamabe soliton is a Yamabe soliton \((M^{n},g,X,\lambda)\) in
  which \(X=\gradientfield<g>{f}\) is the gradient field of some function
  \(f\in{\smoothfunctions(M)}\), called potential function. In this case,
  equation~\eqref{eq:yamabe-soliton-2} becomes
  \begin{equation}\label{eq:yamabe-soliton-3}
    \hessian<g>{f}=(\scalarcurvature-\lambda)g.
  \end{equation}
  The classification of Yamabe solitons according to the sign of \(\lambda\)
  applies also to gradient Yamabe solitons. A gradient Yamabe soliton is
  trivial in case \(\hessian<g>{f}=0\). Finally, we write a gradient Yamabe
  soliton like \((M^{n},g,f,\lambda)\).
\end{definition}

Many efforts have been made in order to classify the geometry of Yamabe solitons
that can be isometrically immersed into space forms. For instance, Chen and
Deshmukh~\cite{chen2018yamabe} studied Yamabe solitons whose soliton field is
the tangent component of the position vector on Euclidean space, classifying
such objects.

Seko and Maeta~\cite{seko2019classification} showed that if the soliton vector
field of a Yamabe soliton hypersurface of an Euclidean space is concurrent then
the soliton must be contained in either a hyperplane or a sphere. Fujii and
Maeta~\cite{fujii2021classification} extended Seko and Maeta's classification
for other gradient Yamabe solitons.

Suh~\cite{suh2023yamabe} completely classifies both Yamabe solitons and gradient
Yamabe solitons on real hypersurfaces of complex the hyperbolic space
\(\mathbb{C}H^{n}\). Tokura et al.~\cite{tokura2020immersion} investigates
gradient almost Yamabe solitons immersed into warped products of the kind
\(\reals\times_{h}M^{n}\) such that the potential function is the height
function with respect to some given direction. As a result, they were able to
classify rotationally symmetric gradient almost Yamabe soliton hypersurfaces of
\(\reals\times_{h}\euclideanspace{n}\).

These works tells us just how deeply the geometry of a Yamabe soliton is
governed by its vector field. Nevertheless, it's desireable to work with
gradient Yamabe soliton hypersurfaces without extra conditions on the soliton
field. This is precisely what we propose with this paper, to study gradient
Yamabe solitons $(M^{n},g,f,\lambda)$ isometrically immersed into space forms
with no further assumptions about the gradient field $\gradientfield{f}$.

The paper is organized as follows. Section~\ref{sec:preliminaries} presents some
examples of gradient Yamabe solitons and establishes the notation to be used
through out the paper. Section~\ref{sec:oy} provides the weak Omori-Yau maximum
principle, proving it to hold true also for gradient Yamabe soliton
hypersurfaces. We use it to determine the sign of the scalar curvature of
gradient Yamabe solitons, both shrinking and steady. Finally,
section~\ref{sec:classification} classifies the gradient Yamabe soliton
hypersurfaces of space forms such that its mean curvature is constant or else
the potential function has at most one critical point.

\section{Preliminaries and examples}\label{sec:preliminaries}

Let \(\ambientspace{c}{n+1}\) be a simply connected Riemannian
\((n+1)\)-manifold of constant sectional curvature \(c\), where
\(c\in{\left\{-1,0,1\right\}}\). In order to make it precise, let
\(\euclideanspace{n+2}_{\nu}\) be the \((n+2)\)-dimensional real vector space
\[
  \euclideanspace{n+2}
  =
  \left\{
    (u_{1},\ldots,u_{n+2})
    :
    u_{1},\ldots,u_{n+2}\in{\reals}
  \right\},
\]
endowed with either the inner product
\[
  \innerproduct{u,v}
  =
  \sum\limits_{i=1}^{n+2}u_{i}v_{i}
  ,
\]
if \(\nu=0\), or else with the non-degenerate symmetric
bilinear form
\[
  \innerproduct{u,v}
  =
  -
  u_{1}v_{1}
  +
  \sum\limits_{i=2}^{n+2}u_{i}v_{i}
  ,
\]
if \(\nu=1\), where \(u=(u_{1},\ldots,u_{n+2})\in{\euclideanspace{n+2}}\) and
\(v=(v_{1},\ldots,v_{n+2})\in{\euclideanspace{n+2}}\). Then, we take
\[
  \begin{array}{lllll}
    \ambientspace{-1}{n+1} & = & \hyperbolicspace{n+1} & \equiv & \left\{u\in{\euclideanspace{n+2}}:\innerproduct{u,u}=-1,\,\innerproduct{u,e_{1}}>0\right\}\subset{\euclideanspace{n+2}_{1}}\\
    \ambientspace{0}{n+1}  & = & \euclideanspace{n+1}  & \equiv & \left\{u\in{\euclideanspace{n+2}}:\innerproduct{u,e_{n+2}}=0\right\}\subset{\euclideanspace{n+2}_{0}}                      \\
    \ambientspace{1}{n+1}  & = & \roundsphere{n+1}     & \equiv & \left\{u\in{\euclideanspace{n+2}}:\innerproduct{u,u}=1\right\}\subset{\euclideanspace{n+2}_{0}}
  \end{array}
\]
all three of them munished with the induced metric from the surrounding space,
where \(e_{1},\ldots,e_{n+2}\in{\euclideanspace{n+2}}\) stand for the canonical
vectors as usual.

Let \(\phi:\submanifold\to{\ambientspace{c}{n+1}}\) be an isometric immersion
of a connected and oriented Riemannian \(n\)-manifold \(\submanifold\) into
$\ambientspace{c}{n+1}$. For all local formulas and computations we may
consider \(\phi\) as an imbedding and thus identify \(x\in{\submanifold}\) with
\(\phi(x)\in{\ambientspace{c}{n+1}}\). The tangent space
\(\tangentspace{x}{\submanifold}\) is identified with a subspace of the tangent
space \(\tangentspace{x}{\ambientspace{c}{n+1}}\) and the normal space
\(\normalspace{x}{\submanifold}\) is the subspace of
\(\tangentspace{x}{\ambientspace{c}{n+1}}\) consisting of all
\(X\in{\tangentspace{x}{\ambientspace{c}{n+1}}}\) which are orthogornal to
\(\tangentspace{x}{\submanifold}\) with respect to the Riemannian metric of
\(\ambientspace{c}{n+1}\).

The second fundamental form \(\secondfundamentalform\) and the shape operator
\(\shapeoperator\) are related to the covariant differentiations
\(\ambientconn\) of \(\ambientspace{c}{n+1}\) and \(\submanifoldconn\) of
\(\submanifold\) by the following formulas:
\begin{equation}\label{eq:second-fundamental-form}
	\ambientconn[X](Y)
  =
  \submanifoldconn[X](Y)
  +
  \secondfundamentalform[X](Y)
  \quad
  \left(\forall{X,Y\in{\smoothvectorfields(\submanifold)}}\right),
\end{equation}
\begin{equation}\label{eq:shape-operator}
  \ambientconn[X](\normalfield)
  =
  -\shapeoperator[X]
  \quad
  \left(\forall{X\in{\smoothvectorfields(\submanifold)}}\right),
\end{equation}
where \(\normalfield\in{\smoothvectorfields(\ambientspace{c}{n+1})}\) is a
field of unit vectors normal to \(\submanifold\). The Gauss equation is given
by
\begin{equation}\label{eq:gauss-equation}
  \curvaturetensor{X}{Y}Z
  =
  c(X\wedge{Y})Z
  -
  (\shapeoperator[X]\wedge{\shapeoperator[Y]})Z,
  \quad
  \left(\forall{X,Y,Z\in{\smoothvectorfields(\submanifold)}}\right),
\end{equation}
where \(X\wedge{Y}\) is the skew-symmetric endomorphism of
\(\smoothvectorfields(\submanifold)\) given by
\[
  \smoothvectorfields(\submanifold)\to{\smoothvectorfields(\submanifold)},
  \quad
  Z\mapsto{(X\wedge{Y})Z=\innerproduct{X,Z}Y-\innerproduct{Y,Z}X}.
\]
The mean curvature of the immersion is
\(\meancurvature=\frac{1}{n}\trace(\shapeoperator)\) and the Ricci curvature of
\(\submanifold\) at a given direction then is
\begin{equation}\label{eq:ricci-curvature}
  \riccitensor[X](X)
  =
  c(n-1)\norm{X}^{2}
  +
  n\meancurvature
  \innerproduct{\shapeoperator[X],X}
  -
  \norm{\shapeoperator[X]}^{2}
  \quad
  \left(\forall{X\in{\smoothvectorfields(\submanifold)}}\right),
\end{equation}
with what we get that the scalar curvature of \(\submanifold\) must be given by
\begin{equation}\label{eq:scalar-curvature}
  \scalarcurvature
  =
  n(n-1)c
  +
  n^{2}\meancurvature^{2}
  -
  \norm{\shapeoperator}^{2}.
\end{equation}
As a final note we would to recall an interesting inequality regarding the
quantities \(\norm{\shapeoperator}\) and \(\meancurvature\). If
\(\tracelessshapeoperator=\shapeoperator-\meancurvature\cdot{\identityoperator}\)
denotes the traceless part of the shape operator, then
\[
  \norm{\tracelessshapeoperator}^{2}
  =
  \norm{\shapeoperator}^{2}-n\,\meancurvature^{2}
  \geqslant{0}
  \implies
  \sqrt{n}\,\abs{\meancurvature}\leqslant{\norm{\shapeoperator}},
\]
with the equality taking place if, and only if, \(\submanifold\) is a totally
umbilical submanifold of \(\ambientspace{c}{n+1}\). We end up this section with
some examples.

\begin{example}[Hyperplanes]
  Let \(\euclideanspace{n}\hookrightarrow\euclideanspace{n+1}\) be an isometric
  immersion. Then, because this is a totally geodesic submanifold with null
  scalar curvature, we get that
  \((\euclideanspace{n},\innerproduct{,},f,\alpha)\) is a gradient Yamabe
  soliton for
  \[
    f(x)
    =
    -
    \frac{\alpha}{2}
    \norm{x}^{2}
    +
    \innerproduct{x,\beta}
    +
    \gamma,
  \]
  for any constants \(\alpha,\gamma\in{\reals}\) and
  \(\beta\in{\euclideanspace{n}}\). When \(\beta=\gamma=0\) this is known as
  the \textit{Gaussian} soliton.
\end{example}

\begin{example}[Horospheres]
  Let \(\euclideanspace{n}\hookrightarrow\hyperbolicspace{n+1}\) be an
  isometric immersion of a horosphere into \(\hyperbolicspace{n+1}\). Then,
  because \(\shapeoperator=\identityoperator\) we also get that
  \((\euclideanspace{n},\innerproduct{,},f,\alpha)\) is a gradient Yamabe
  soliton for
  \[
    f(x)
    =
    -
    \frac{\alpha}{2}
    \norm{x}^{2}
    +
    \innerproduct{x,\beta}
    +
    \gamma,
  \]
  for any constants \(\alpha,\gamma\in\reals\) and
  \(\beta\in\euclideanspace{n}\).
\end{example}

\begin{example}[Cylinders]
  Let
  \(\euclideanspace{k}\times\roundsphere{n-k}(r)\hookrightarrow\euclideanspace{n+1}\)
  be an isometric immerion of a cylinder of radius \(r>0\) into flat Euclidean
  \((n+1)\)-space. Then, because this cylinder has principal curvatures
  \[
    \kappa_{1}=\cdots=\kappa_{k}=0
    \quad
    \text{and}
    \quad
    \kappa_{k+1}=\cdots=\kappa_{n}=\frac{1}{r}
    ,
  \]
  we get that
  \[
    \meancurvature
    =
    \frac{n-k}{nr}
    \qquad
    \text{and}
    \qquad
    \norm{\shapeoperator}^{2}
    =
    \frac{n-k}{r^2}
    .
  \]
  Therefore, the tuple
  \((\euclideanspace{k}\times\roundsphere{n-k}(r),\innerproduct{,},f,\lambda)\)
  is a trivial gradient Yamabe soliton for
  \[
    f:\euclideanspace{k}\times\roundsphere{n-k}(r)\to{\reals},
    \quad
    (x,y)\mapsto{\innerproduct{\alpha,x}}+\beta,
  \]
  for whatever constants \(\alpha\in{\euclideanspace{k}}\) and
  \(\beta\in\reals\) that we take, as long as we set
  \(\lambda=\scalarcurvature=n^{2}\meancurvature^{2}-\norm{\shapeoperator}^{2}\).
\end{example}

\begin{example}[Products of spheres]
  If
  \(\roundsphere{k}(\sqrt{1-r^{2}})\times\roundsphere{n-k}(r)\hookrightarrow\euclideanspace{n+1}\)
  is an isometric immersion of a Riemannian product of spheres into flat
  Euclidean \((n+1)\)-space, where \(r\in\ointerval{0}{1}\), we then get that
  \[
    \kappa_{1}
    =
    \cdots
    =
    \kappa_{k}
    =
    \frac{r}{\sqrt{1-r^{2}}}
    \quad
    \text{and}
    \quad
    \kappa_{k+1}
    =
    \cdots
    =
    \kappa_{n}
    =
    -
    \frac{\sqrt{1-r^{2}}}{r}
    ,
  \]
  are the principal curvatures for this immersion. Therefore, we conclude that
  \[
    \meancurvature
    =
    \frac{nr^{2}-(n-k)}{nr\sqrt{1-r^{2}}}
    \quad
    \text{and}
    \quad
    \norm{\shapeoperator}^{2}
    =
    \frac{kr^{4}+(n-k)(1-r^{2})^{2}}{r^{2}(1-r^{2})}
    .
  \]
  Since compact gradient Yamabe solitons of dimension \(n\geqslant{3}\) must be
  trivial (see~\cite{hsu2012note}), the tuple
  \((\roundsphere{k}(\sqrt{1-r^{2}})\times\roundsphere{n-1}(r),\innerproduct{,},f,\lambda)\)
  satisfies equation~\eqref{eq:yamabe-soliton-2} for a constant \(f\), with
  \(\lambda=\scalarcurvature=n^{2}\meancurvature^{2}-\norm{\shapeoperator}^{2}\).
\end{example}

\begin{example}[Products of hyperbolic spaces and circles]\label{example1}
  Let \(r>0\). If
  \(\hyperbolicspace{k}(-\sqrt{1+r^{2}})\times\roundsphere{n-k}\hookrightarrow\hyperbolicspace{n+1}\)
  is an isometric immersion of the Riemannian product of a hyperbolic
  \(k\)-space and an \((n-k)\)-sphere into the hyperbolic \((n+1)\)-space, then
  we get
  \[
    \kappa_{1}
    =
    \cdots
    =
    \kappa_{k}
    =
    \frac{r}{\sqrt{1+r^{2}}}
    \quad
    \text{and}
    \quad
    \kappa_{k+1}
    =
    \cdots
    \kappa_{n}
    =
    \frac{\sqrt{1+r^{2}}}{r}.
  \]
  for the principal curvatures of this immersion, and it then follows that
  \[
    \meancurvature
    =
    \frac{nr^{2}+(n-k)}{nr\sqrt{1+r^{2}}}
    \quad
    \text{and}
    \quad
    \norm{\shapeoperator}^{2}
    =
    \frac{kr^{4}+(n-k)(1+r^{2})^{2}}{r^{2}(1+r^{2})}
    .
  \]
  Therefore, for a constant $f$, the tuple
  \((\hyperbolicspace{k}(-\sqrt{1+r^{2}})\times\roundsphere{n-k},\innerproduct{,},f,\lambda)\)
  is a trivial gradient Yamabe soliton with
  \(\lambda=\scalarcurvature=-n(n-1)+n^{2}\meancurvature^{2}-\norm{\shapeoperator}^{2}\).
  This soliton is steady if \((n,k)=(2,1)\) and expanding if \((n,k)=(n,n-1)\)
  where \(n>2\). At this point we would like to propose the following question:
  could we have made a gradient Yamabe soliton out of
  \(
    \hyperbolicspace{n-1}(-\sqrt{1+r^{2}})\times\roundsphere{1}
  \)
  in any other way? For example, could we have chosen \(f\) in some class of
  smooth functions other than the constant ones? It follows from Lemma~\ref{thb}
  that, for \(n>2\), the answer is negative for gradient Yamabe solitons of the
  steady and shrinking kinds.
\end{example}

\section{The weak Omori-Yau maximum principle on gradient Yamabe solitons hypersurfaces}\label{sec:oy}

We utilize the Omori-Yau maximum principle to prove our results. Following the
terminology introduced by Pigola, Rigoli and Setti in~\cite{pigola2005maximum},
and also by Alías, Mastrolia and Rigoli in~\cite{alias2016maximum}, the weak
Omori-Yau maximum principle holds for the drifted Laplacian
\(\Delta_{w}=e^w\textup{div}(e^{-w}\nabla)\) on a Riemannian manifold
\((M^{n},g)\) if, for any function \(u\in{C^{2}(M)}\) with
\(u^{\ast}=\sup_{M}u<+\infty\), there exists a sequence of points
\(\{x_{k}\}_{k\in\naturals}\) in \(M\) such that
\[
  \driftedlaplacian{w}{u}(x_{k})<\frac{1}{k}
  \quad\textup{and}\quad
  u(x_{k})>u_{\ast}-\frac{1}{k},
\]
for each \(k\in\naturals\). The weak Omori-Yau maximum has been shown to be true
(see Corollary 8.3 of~\cite{alias2016maximum}) on every complete
\(n\)-dimensional Riemannian manifold satisfying
\begin{equation*}
  \riccitensor[\gradientfield{r}](\gradientfield{r})
  +
  \hessian{w}[\gradientfield{r}](\gradientfield{r})
  \geqslant
  -(n-1)G(r),
\end{equation*}
as well as \(\norm{\gradientfield{w}}^{2}\leqslant{F(r)}\), where
\(r:M\to\rinterval{0}{\infty}\) is the Riemannian distance from a given point,
and \(F,G\) are real \(C^{1}\)-functions on \(\rinterval{0}{\infty}\) with \(F\)
nondecreasing and
\begin{align*}
  &
  \inf_{\rinterval{0}{\infty}}\dfrac{G'}{G^{3/2}}>-\infty,
  \quad
  \frac{1}{\sqrt{F}+\sqrt{G}}\notin{L^{1}(\ointerval{0}{\infty})},
  \\
  &
  \left(\sqrt{F}+\sqrt{G}\right)'(t)\geqslant{-B(\log{t+1})},
\end{align*}
for all large enough \(t\) and some constant \(B\in\rinterval{0}{\infty}\).

It has been brought to our attention by the work~\cite{albanese2013general} of
Albanese, Alías and Rigoli that the validity of the weak Omori–Yau maximum
principle on \((M^{n},g)\) does not depend on curvature bounds as much as one
would expect. For instance, the weak Omori-Yau maximum principle for
\(\Delta_{w}\) remains true on every complete \(n\)-dimensional Riemannian
manifold admitting a non-negative \(C^{2}\)-function \(r\) on \(M\) satisfying
the following requirements:
\begin{enumerate}[a)]
  \item
    \(\lim\limits_{x\to{\infty}}r(x)=\infty\);
  \item
    \(\Delta_{w}r(x)\leqslant G(r(x))\) outside of a compact subset of \(M\),
\end{enumerate}
where \(G\) is a non-negative \(C^{1}\)-function on \(\ointerval{0}{\infty}\)
satisfying
\[
  \dfrac{1}{G}
  \notin
  L^{1}
  \quad\textup{and}\quad
  G'(t)
  \geqslant
  -B(\log t+1),
\]
for all large enough \(t\) and some constant \(B\in\rinterval{0}{\infty}\). The
next result states that the weak Omori-Yau maximum principle remains true for
\(\Delta_{f/[2(n-1)]}\) on \((M^{n},g)\) provided that \(f\) is the potential
function of a gradient Yamabe soliton \((M^{n},g,f,\lambda)\), either shrinking
or steady.

\begin{lemma}\label{tha}
  Let \((M^{n},g,f,\lambda)\) be a complete gradient Yamabe soliton immersed
  into a space form \(\ambientspace{c}{n+1}\). Then, the weak Omori-Yau maximum
  principle holds true for \(\Delta_{\frac{f}{2(n-1)}}\) provided that:
  \begin{enumerate}[a)]
    \item
      \(n=2\);
    \item
      \begin{enumerate}[i.]
        \item
          \(n>2\)
        \item
          \(\sup\limits_{M^{n}}\norm{\shapeoperator}<\infty\).
      \end{enumerate}
  \end{enumerate}
\end{lemma}

\begin{proof}[Proof of Lemma~\ref{tha}]
  Let \(n>2\). Take some \(x_{0}\in{M^{n}}\) and let \(r(x)=d(x,x_{0})\) be the
  Riemannian distance from \(x\) to \(x_{0}\). Since \(M^{n}\) is complete,
  there exists a minimizing geodesic \(\alpha:\interval{0}{r(x)}\to{M^{n}}\)
  from \(x_{0}\) to \(x\). Without loss of generality, we may assume that \(r\)
  is smooth at \(x\). Let \(\mathcal{B}=\{e_{1},\ldots,e_{n}\}\) be an
  orthonormal basis at \(x_{0}\) with \(\alpha'(0)=e_{1}\) and then take
  \(e_{1}(s),\ldots,e_{n}(s)\) to be the parallel displacement of the
  \(e_{1},\ldots,e_{n}\) from \(x_{0}\) up to \(\alpha(s)\). For each
  \(j\in{\left\{1,\ldots,n\right\}}\), let \(X_{j}(s)\) be the Jacobian field
  along \(\alpha(s)\) which is given by:
  \[
    \left\{
    \begin{array}{rcl}
      X_{j}(0)    & = & 0, \\
      X_{j}(r(x)) & = & e_{j}(r(x)).
    \end{array}
    \right.
  \]
  Then, we get that
  \[
    \laplacian{r}(x)
    =
    \sum_{j=1}^{n}
    \int_{0}^{r(x)}
    \left(
    \norm{X_{j}'(s)}^{2}
    -
    R(\alpha'(s),X_{j}(s),\alpha'(s),X_{j}(s))
    \right)
    ds.
  \]
  Let's take some \(r_{0}\in{\ointerval{0}{r(x)}}\) such that
  \[
    \norm{\riccitensor}\leqslant{(n-1)\,r_{0}^{-2}}
    \quad\textup{and}\quad
    \norm{\gradientfield{f}}(x_{0})\leqslant{(n-1)\,r_{0}^{-1}},
  \]
  on \(B(x_{0},r_{0})\subset{M^{n}}\) and define vector fields
  \[
    Y_{1}(s)=X_{1}(s)
    \quad\textup{and}\quad
    Y_{j}(s)=c(s)e_{j}(s)
    \quad(j\geqslant{2}),
  \]
  along \(\alpha:\interval{0}{r(x)}\to{M}\), where
  \(c:\interval{0}{r(x)}\to\reals\) is given by
  \begin{equation*}
    c(s) =
    \begin{cases}
      r_{0}^{-1}s, & \text{if }s\in\interval{0}{r_{0}}, \\
      1,           & \text{if }s\in\interval{r_{0}}{r(x)}.
    \end{cases}
  \end{equation*}
  It follows from the standard index comparison theorem that
  \begin{equation*}
    \begin{split}
      \laplacian{r}(x)
      &
      \leqslant
      \sum_{j=1}^{n}
        \int_{0}^{r(x)}
        \left(
          \norm{Y_{j}'(s)}^{2}
          -
          R(\alpha'(s),Y_{j}(s),\alpha'(s),Y_{j}(s))
        \right)
        ds
      \\
      &
      \leqslant
      \int_{0}^{r_{0}}
        \left[
          \frac{n-1}{r_{0}^{2}}
          +
          \left(
            1-\frac{s^{2}}{r_{0}^{2}}
          \right)
          \riccitensor[\alpha'(s)](\alpha'(s))
        \right]
      ds
      \\
      &
      \hspace{2.6cm}
      -
      \int_{0}^{r(x)}
      \riccitensor[\alpha'(s)](\alpha'(s))
      ds
      \\
      & \leqslant
      \frac{n-1}{r_{0}}
      +
      \frac{2}{3}\frac{(n-1)^{2}}{r_{0}}
      -
      \int_{0}^{r(x)}
      \riccitensor[\alpha'(s)](\alpha'(s))
      ds.
    \end{split}
  \end{equation*}
  Also, we get from the Yamabe soliton equation~\eqref{eq:yamabe-soliton-1} that
  \begin{equation*}
    \begin{split}
      g(\gradientfield{f},\gradientfield{r})(x)
      & =
      \int_{0}^{r(x)}
      \frac{d}{ds}
      g(\gradientfield{f},\gradientfield{r})
      ds
      \\
      &
      \hspace{2.5cm}
      +
      g(\gradientfield{f},\gradientfield{r})(x_{0})
      \\
      & =
      \int_{0}^{r(x)}
      \hessian{f}[\alpha'(s)](\alpha'(s))
      ds
      \\
      &
      \hspace{2.5cm}
      +
      g(\gradientfield{f},\gradientfield{r})(x_{0})
      \\
      & \leqslant
      \int_{0}^{r(x)}
      \scalarcurvature
      ds
      +
      \norm{\gradientfield{f}}(x_{0})
      -
      \lambda{r(x)}.
    \end{split}
  \end{equation*}
  Then, we get that
  \begin{equation}\label{estimativa}
    \begin{split}
      &
      2(n-1)\laplacian{r}(x)
      +
      g(\gradientfield{f},\gradientfield{r})(x)
      \\
      &
      \hspace{0.5cm}
      \leqslant
      \frac{C_{n}}{r_{0}}
      -
      2(n-1)
      \int_{0}^{r(x)}
      \riccitensor[\alpha'(s)](\alpha'(s))
      ds
      +
      \int_{0}^{r(x)}
      \scalarcurvature
      ds
      -
      \lambda{r(x)},
    \end{split}
  \end{equation}
  for every \(x\in{M^{n}\setminus{B(x_{0},r_{0})}}\), where \(C_{n}\) is a
  constant that depends on \(n\).

  %\(3\,C_{n}={10(n-1)^{2}+3(n-1)}\).

  On the other hand, it follows from~\eqref{eq:ricci-curvature}
  and~\eqref{eq:scalar-curvature} that
  \begin{equation*}
    \begin{split}
      &
      2(n-1)
      \laplacian{r}(x)
      +
      g(\gradientfield{f},\gradientfield{r})(x)
      \\
      &
      \hspace{0.5cm}
      \leqslant
      \frac{C_{n}}{r_{0}}
      -
      2(n-1)
      \int_{0}^{r(x)}
      \riccitensor[\alpha'(s)](\alpha'(s))
      ds
      +
      \int_{0}^{r(x)}
      \scalarcurvature
      ds
      -
      \lambda{r(x)}
      \\
      &
      \hspace{0.5cm}
      \leqslant
      \frac{C_{n}}{r_{0}}
      -
      (n-1)(n-2)cr(x)
      +
      \int_{0}^{r(x)}
      \left(
        n^{2}\meancurvature^{2}
        -
        \norm{\shapeoperator}^{2}
        \right)
      ds
      -
      \lambda{r(x)}
      \\
      &
      \hspace{1.0cm}
      -
      \int_{0}^{r(x)}
      \left(
        2n(n-1)
        \meancurvature
        g(\shapeoperator[\alpha'(s]),\alpha'(s))
        -
        2(n-1)\norm{\shapeoperator\alpha'(s)}^{2}
      \right)
      ds
      \\
      &
      \hspace{0.5cm}
      \leqslant
      \frac{C_{n}}{r_{0}}
      -
      (n-1)(n-2)cr(x)
      +
      \int_{0}^{r(x)}
      \left(
        n^{2}\meancurvature^{2}
        -
        \norm{\shapeoperator}^{2}
        \right)
      ds
      -
      \lambda{r(x)}
      \\
      &
      \hspace{1.0cm}
      +
      \int_{0}^{r(x)}
      \left(
        2n(n-1)
        \abs{\meancurvature}
        \norm{\shapeoperator}
        +
        2(n-1)
        \norm{\shapeoperator}^{2}
        \right)
      ds
      \\
      &
      \hspace{0.5cm}
      \leqslant
      \frac{C_{n}}{r_{0}}
      -
      (n-1)(n-2)cr(x)
      +
      \int_{0}^{r(x)}
      \left(
        n^{2}\meancurvature^{2}
        -
        \norm{\shapeoperator}^{2}
        \right)
      ds
      -
      \lambda{r(x)}
      \\
      &
      \hspace{1.0cm}
      +
      \int_{0}^{r(x)}
      \left(
        n(n-1)\meancurvature^{2}
        +
        n(n-1)
        \norm{\shapeoperator}^{2}
        +
        2(n-1)
        \norm{\shapeoperator}^{2}
      \right)
      ds
      \\
      &
      \hspace{0.5cm}
      \leqslant
      \frac{C_{n}}{r_{0}}
      -
      (n-1)(n-2)cr(x)
      +
      (n-1)(n+4)
      \sup\limits_{M}\norm{\shapeoperator}^{2}
      r(x)
      -
      \lambda{r(x)}
      \\
      &
      \hspace{0.5cm}
      \leqslant
      \frac{C_{n}}{r_{0}}+\abs{K}r(x).
    \end{split}
  \end{equation*}
  where
  \[
    K
    =
    -(n-1)(n-2)c
    +
    (n-1)(n+4)\sup\limits_{M^{n}}{\norm{\shapeoperator}^{2}}
    -
    \lambda.
  \]
  Hence, from Theorem A and Remark 1.1 of~\cite{albanese2013general}, the weak
  Omori-Yau maximum principle holds for the drifted Laplacian
  \(\Delta_{f/2(n-1)}\).

  Now, let \(n=2\). Then, because we have
  \[
    \riccitensor[X](X)=\frac{\scalarcurvature}{2}g(X,X)
    \quad(\forall{X\in\smoothvectorfields(M)}),
  \]
  the inequality~\eqref{estimativa} takes the form
  \begin{equation*}
    2(n-1)
    \laplacian{r}(x)
    +
    g(\gradientfield{f},\gradientfield{r})(x)
    \leqslant
    \frac{C_{n}}{r_{0}}.
  \end{equation*}
  Then, once again by Theorem A of~\cite{albanese2013general}, the weak
  Omori-Yau maximum principle holds for the drifted Laplacian \(\Delta_{f/2}\).
\end{proof}

\begin{lemma}[\cite{chow1992yamabe,daskalopoulos2013classification,ma2012remarks}]
  Let \((M^{n},g,f,\lambda)\) be a gradient Yamabe soliton. Then
  \begin{equation*}
    2(n-1)
    \laplacian{\scalarcurvature}
    +
    g(\gradientfield{f},\gradientfield{\scalarcurvature})
    +
    2\scalarcurvature(\scalarcurvature-\lambda)
    =
    0.
  \end{equation*}
  \label{Lemma1}
\end{lemma}

As a consequence of Lemma~\ref{tha}, we shall prove the following result on the
sign of the scalar curvature on gradient Yamabe solitons.

\begin{lemma}\label{thb}
  Let \((M^{n},g,f,\lambda)\) be a complete, shrinking or steady, gradient
  Yamabe soliton isometrically immersed into a space form
  \(\ambientspace{c}{n+1}\). Assume that
  \(\sup\left\{\norm{A_{x}}:x\in{M^{n}}\right\}\) is finite whenever \(n>2\)
  (if \(n=2\) this requirement is not necessary). Then, the scalar curvature of
  \((M^{n},g)\) satisfies
  \[
    0
    \leqslant
    \inf_{M^{n}}\scalarcurvature
    \leqslant
    \lambda.
  \]
  Furthermore, if the above infimum is attained at some point
  \(x_{0}\in{M^{n}}\) and one verifies that either \(\scalarcurvature(x_{0})=0\)
  or \(\scalarcurvature(x_{0})=\lambda\), then \(\scalarcurvature\) must be
  constant on \(M^{n}\).
\end{lemma}

\begin{proof}[Proof of Theorem~\ref{thb}]
  Take both \(-u=\scalarcurvature\) and \(2(n-1)w=f\). Then, we get from
  Lemma~\ref{Lemma1} that
  \begin{equation*}
    \driftedlaplacian{w}{u}=\frac{\lambda{u}}{n-1}+\frac{u^{2}}{n-1}.
  \end{equation*}
  By applying Theorem 3.2 of~\cite{albanese2013general} with
  \[
    F(t)
    =
    t^{2}
    \quad\textup{and}\quad
    \varphi(u,\norm{\gradientfield{u}})
    =
    \frac{\lambda{u}}{n-1}+\frac{u^{2}}{n-1},
  \]
  we get that \(u^{\ast}=\sup\left\{u(x):x\in{M^{n}}\right\}<\infty\). Also, by
  Lemma~\ref{tha}, there exists a sequence
  \(\{x_{k}\}_{k\in\naturals}\subset{M^{n}}\) such that
  \[
    \driftedlaplacian{w}{u}(x_{k})
    <
    \frac{1}{k}
    \quad\textup{and}\quad
    u(x_{k})
    >
    u^{\ast}
    -
    \frac{1}{k}
    \quad(\forall{k\in{\naturals}}).
  \]
  Thus, we get that
  \[
    \frac{\lambda{u}^{\ast}}{n-1}
    +
    \frac{(u^{\ast})^{2}}{n-1}
    \leqslant{0},
  \]
  and because \(u=-\scalarcurvature\) we have that
  \[
    u^{\ast}
    =
    -\inf_{M^{n}}\scalarcurvature
    =
    -\scalarcurvature_{\ast},
  \]
  with \(0\leqslant{\scalarcurvature_{\ast}}\leqslant\lambda\). Therefore,
  \(\scalarcurvature\geqslant{0}\) on \(M^{n}\). Now, suppose that we have
  \(\scalarcurvature(x_{0})=-u(x_{0})=0\) for some \(x_{0}\in{M^{n}}\) and let
  \[
    \Omega_{0}=\{x\in{M^{n}}:u(x)=0\}=u^{-1}(0).
  \]
  Then, \(\Omega_{0}\) is a nonempty and closed subset of \(M^{n}\).
  Now, given any \(y\in{\Omega_{0}}\), since
  \begin{equation}\label{abc}
    \driftedlaplacian{w}{u}
    -
    \frac{\lambda{u}}{n-1}
    =
    \frac{u^{2}}{n-1}
    \geqslant{0}.
  \end{equation}
  we get from the maximum principle (see~\cite{gilbarg2015elliptic} p. 35) that
  \(u(x)=0\) in some neighborhood of \(y\), implying that \(\Omega_{0}\) is
  open. Since \(M^{n}\) is connected, this gives \(\Omega_{0}=M^{n}\), meaning
  that \(\scalarcurvature=0\) on \(M^{n}\). If
  \(\scalarcurvature(x_{0})=\lambda\) for some \(x_{0}\in{M^{n}}\), then one can
  reason in a similar way to show that we must have \(\scalarcurvature=\lambda\)
  on \(M^{n}\).
\end{proof}

\section{Classification of gradient Yamabe solitons hypersurfaces}\label{sec:classification}

The study of constant mean curvature hypersurfaces in space forms is one of the
oldest subjects in differential geometry. A well known result of Klotz and
Osserman~\cite{klotz1966complete} states that the standard spheres and the
circular cylinders are the only complete surfaces in the Euclidean space
\(\euclideanspace{3}\) with nonzero constant mean curvature whose Gaussian
curvature does not change sign. Hoffman~\cite{hoffman1973surfaces} and
Tribuzy~\cite{tribuzy1978hopf} gave an extension of that result to the case of
surfaces with constant mean curvature in the Euclidean 3-sphere
\(\roundsphere{3}\) and in the hyperbolic space \(\hyperbolicspace{3}\),
respectively. Summarizing the results of those authors in a single statement,
we obtain the following result.

\begin{proposition}[\cite{alias2010scalar,alias2012estimate}]\label{prop1}
  Let \(M^{2}\) be a complete surface isometrically immersed into a
  \(3\)-dimensional space form \(\ambientspace{c}{3}\) with constant mean
  curvature \(\meancurvature\). If its Gaussian curvature \(K\) does not change
  sign, then \(M^{2}\) is either a totally umbilical surface or else \(K=0\) and
  also:
  \begin{enumerate}[a)]
    \item
      \(M^{2}\) is the cylinder
      \(\euclideanspace{1}\times\roundsphere{1}(r)\subset\euclideanspace{3}\)
      with \(r>0\) if \(c=0\).
    \item
      \(M^{2}\) is the flat torus
      \(\roundsphere{1}(\sqrt{1-r^{2}})\times\roundsphere{1}(r)\subset\roundsphere{3}\)
      with \(0<r<1\) if \(c=1\).
    \item
      \(M^{2}\) is the hyperbolic cylinder
      \(\hyperbolicspace{1}(-\sqrt{1+r^{2}})\times\roundsphere{1}(r)\subset\hyperbolicspace{3}\)
      with \(r>0\) if \(c=-1\).
  \end{enumerate}
\end{proposition}

As a consequence of Lemma~\ref{thb} and Proposition~\ref{prop1}, we deduce the
following classification of gradient Yamabe soliton surfaces, either shrinking
or steady, isometrically immersed into \(3\)-dimensional space forms.

\begin{theorem}\label{theoremnew1}
  Let \((M^{2},g,f,\lambda)\) be a complete \(2\)-dimensional gradient Yamabe
  soliton, shrinking or steady, isometrically immersed with constant mean
  curvature into a \(3\)-dimensional space form \(\ambientspace{c}{3}\). Then
  \(M^{2}\) is either totally umbilical, or
  \begin{enumerate}[a)]
    \item
      \(M^{2}\) is the cylinder
      \(\euclideanspace{1}\times\roundsphere{1}(r)\subset\euclideanspace{3}\)
      with \(r>0\).
    \item
      \(M^{2}\) is the flat torus
      \(\roundsphere{1}(\sqrt{1-r^{2}})\times\roundsphere{1}(r)\subset\roundsphere{3}\)
      where \(0<r<1\). In this case, a trivial soliton since \(M^{2}\) is
      compact.
    \item
      \(M^{2}\) is the hyperbolic cylinder
      \(\hyperbolicspace{1}(-\sqrt{1+r^{2}})\times\roundsphere{1}(r)\subset\hyperbolicspace{3}\)
      with \(r>0\).
  \end{enumerate}
\end{theorem}

\begin{proof}[Proof of Theorem~\ref{theoremnew1}]
  It follows from the Gauss equation in conjunction with Lemma~\ref{thb} that
  \[
    0
    \leqslant
    \inf_{M^{2}}K
    \leqslant
    K
    =
    c+k_{1}k_{2}
    \leqslant{c+\meancurvature^{2}},
  \]
  with the equality
  \(\inf\left\{K(x):x\in{M^{2}}\right\}=c+\meancurvature^{2}\) taking place if
  and only if \(M^{2}\) is totally umbilical. If \(c+\meancurvature^{2}=0\),
  then \(c+\meancurvature^{2}=\inf\left\{K(x):x\in{M^{2}}\right\}=0\) and
  \(M^{2}\) is totally umbilical. If \(c+\meancurvature^{2}>0\), then
  \(\inf\left\{K(x):x\in{M^{2}}\right\}\leqslant{0}\) because otherwise we
  would get \(K\geqslant{\inf\left\{K(x):x\in{M^{2}}\right\}}>0\), thus
  contradicting Proposition~\ref{prop1}. Now,
  \(\inf\left\{K(x):x\in{M^{2}}\right\}=0\) and so \(K\) does not change its
  sign on \(M^{2}\) and the proof follows once again from
  Proposition~\ref{prop1}.
\end{proof}

\begin{remark}
  From Theorem 1 of~\cite{pigola2011remarks} and Theorem~\ref{theoremnew1} above
  we conclude that any complete shrinking \(2\)-dimensional gradient Yamabe
  soliton isometrically immersed with constant mean curvature into a
  \(3\)-dimensional space form \(\ambientspace{c}{3}\) is either trivial or
  totally umbilical. Also, because gradient Ricci solitons and gradient Yamabe
  solitons are the same object in dimension \(n=2\), Theorem~\ref{theoremnew1}
  is a classification of complete gradient Ricci solitons, shrinking or steady,
  isometrically immersed with constant mean curvature into \(3\)-dimensional
  space forms.
\end{remark}

Following is an application of Theorem~\ref{theoremnew1} to the case of minimal
surfaces, that is, those with mean curvature \(\meancurvature=0\).

\begin{corollary}\label{coro1}
  Let \((M^{2},g,f,\lambda)\) be a complete \(2\)-dimensional gradient Yamabe
  soliton, shrinking or steady, isometrically immersed as a minimal submanifold
  of a \(3\)-dimensional space form \(\ambientspace{c}{3}\). Then, we have the
  following:
  \begin{enumerate}[a)]
    \item
      \(M^{2}\) is a plane in \(\euclideanspace{3}\) if \(c=0\);
    \item
      \(M^{2}\) is either the Clifford torus
      \(\roundsphere{1}(\sqrt{1/2})\times\roundsphere{1}(\sqrt{1/2})\) or a
      totally geodesic sphere of \(\roundsphere{3}\) if \(c=1\);
    \item
      There does not exist such an \(M^{2}\) in \(\hyperbolicspace{3}\) if
      \(c=-1\).
  \end{enumerate}
\end{corollary}

The next result extend Theorem~\ref{theoremnew1} for gradient Yamabe solitons of
dimension \(n>2\).

\begin{theorem}\label{theoremnew2}
  Let \((\submanifold,g,f,\lambda)\), \(n>2\), be a complete gradient Yamabe
  soliton, either shrinking or steady, isometrically immersed with constant
  mean curvature into a space form \(\ambientspace{c}{n+1}\). Assume that
  \(\sup\left\{\norm{\shapeoperator_{x}}:x\in{\submanifold}\right\}\) is
  finite. Then, either we have that \(\submanifold\) is totally umbilical or
  else
  \begin{equation}\label{ast2}
    0
    \leqslant
    \inf_{\submanifold}\scalarcurvature
    \leqslant
    \min\{\lambda,\mu(\meancurvature,c)\},
  \end{equation}
  where
  \[
    \mu(\meancurvature,c)
    =
    \frac{n(n-2)}{2(n-1)}
    \left(
    2(n-1)c
    +
    n\meancurvature^{2}
    +
    \abs{\meancurvature}
    \sqrt{n^{2}\meancurvature^{2}+4(n-1)c}
    \right).
  \]
  Moreover,
  \begin{enumerate}
    \item
      Assume that the infimum
      \(\inf\left\{\scalarcurvature(x):x\in{\submanifold}\right\}\) is attained
      at some point \(x_{0}\in{\submanifold}\). Then,
      \(\scalarcurvature(x_{0})=\tau\) if and only if \(\scalarcurvature\) is
      equal to \(\tau\) everywhere on \(\submanifold\), where \(\tau=0\) or
      \(\tau=\lambda\);
    \item
      There exists a \(x_{0}\in{\submanifold}\) such that
      \[
        \scalarcurvature(x_{0})
        =
        \inf_{\submanifold}\scalarcurvature
        =
        \mu(\meancurvature,c),
      \]
      if, and only if, \(\submanifold\) is one of the following manifolds:
      \begin{enumerate}
        \item
          A circular cylinder
          \(\euclideanspace{1}\times\roundsphere{n-1}(r)\subset\euclideanspace{n+1}\),
          where \(r>0\).
        \item
          A minimal Clifford torus
          \[
            \roundsphere{k}(\sqrt{k/n})\times\roundsphere{n-k}(\sqrt{(n-k)/n})
            \subset
            \roundsphere{n+1},
          \]
          where \(k=1,\ldots,n-1\), or a constant mean curvature torus
          \[
            \roundsphere{1}(\sqrt{1-r^{2}})\times\roundsphere{n-1}(r)
            \subset
            \roundsphere{n+1},
          \]
          where \(0<r<\sqrt{(n-1)/n}\).
      \end{enumerate}
  \end{enumerate}
\end{theorem}

\begin{proof}[Proof of Theorem \ref{theoremnew2}]
  The proof of Theorem~\ref{theoremnew2} is an application of Lemma~\ref{thb}
  and a result based on a Simons type formula, due to Alías and
  Martínez~\cite{alias2010scalar,alias2012estimate,nomizu1969formula}, for the
  Laplacian of \(\norm{\tracelessshapeoperator}^{2}\). Combining
  Lemma~\ref{thb} and the Gauss equation we conclude that
  \begin{equation*}
    n(n-1)(c+\meancurvature^{2})
    \geqslant
    \norm{\tracelessshapeoperator}^{2}
    \geqslant{0}.
  \end{equation*}
  If \(c+\meancurvature^{2}=0\), then
  \(\norm{\tracelessshapeoperator}^{2}\equiv0\) and \(\submanifold\) is totally
  umbilical. If \(c+\meancurvature^{2}>0\), then \(\submanifold\) is totally
  umbilical and either
  \[
    \inf_{\submanifold}\scalarcurvature
    =
    n(n-1)(c+\meancurvature^{2})
    \quad\text{or}\quad
    \inf_{\submanifold}\scalarcurvature
    \leqslant
    \mu(\meancurvature,c),
  \]
  by Corollary 4 of~\cite{alias2010scalar}. In the latter case, since
  \(0\leqslant\inf\left\{\scalarcurvature(x):x\in{\submanifold}\right\}\leqslant\lambda\),
  we get the estimate in~\eqref{ast2}.
\end{proof}

In particular, we get the following result for gradient Yamabe solitons
immersed as minimal hypersurfaces of the hyperbolic space
\(\hyperbolicspace{n+1}\). Notice that it extends item (ii) of Theorem 1
of~\cite{cunha2021immersions} as we don't impose any conditions on the gradient
field \(\gradientfield{f}\).

\begin{corollary}
  No shrinking or steady gradient Yamabe soliton \((\submanifold,g,f,\lambda)\)
  can be isometrically immersed as a minimal hypersurface of the hyperbolic
  space \(\hyperbolicspace{n+1}\) in such a way that the norm of its second
  fundamental form is bounded from above.
\end{corollary}

The next result states that the hyperplanes are the only shrinking or steady
gradient Yamabe solitons that can be isometrically immersed as minimal
submanifolds of flat Euclidean space \(\euclideanspace{n+1}\) as to have a
bound from above on the norm of its second fundamental form (see item (iii)
Theorem 1 of~\cite{cunha2021immersions}).

\begin{corollary}\label{coro2}
  The only shrinking or steady gradient Yamabe solitons
  \((\submanifold,g,f,\lambda)\) that can be isometrically immersed into
  Euclidean space \(\euclideanspace{n+1}\) as minimal hypersurfaces in such a
  way that the norm of its second fundamental form is bounded from above are
  the hyperplanes.
\end{corollary}

\begin{remark}
  Combining item a) of Corollary~\ref{coro1} and Corollary~\ref{coro2} we show
  that Corollary 1.2 of Mastrolia, Rigoli and Rimoldi~\cite{mastrolia2013some}
  is valid also for gradient Yamabe solitons.
\end{remark}

Cao, Sun and Zhang proved in~\cite{cao2012structure} that nontrivial complete
gradient Yamabe solitons \((\submanifold,g,f,\lambda)\) are rigid in the sense
that either \(f\) has one critical point and \((\submanifold,g)\) is a warped
product manifold of the form
\[
  (\rinterval{0}{\infty},dr^{2})
  \times_{\norm{\gradientfield{f}}}
  (\roundsphere{n-1},g_{\roundsphere{n-1}}),
\]
or else \(f\) has no critical points, and \((\submanifold,g)\) is a warped
product of the form
\[
  (\reals,dr^{2})
  \times_{\norm{\gradientfield{f}}}
  (N^{n-1},\bar{g}),
\]
where \((N^{n-1},\bar{g})\) is a Riemannian manifold of constant scalar
curvature. This motivated us to also pursue a classification in terms of the
number of critical points of the potential function of a gradient Yamabe
soliton that immersed as hypersurface of a space form.

\begin{theorem}\label{theoremx}
  Let \((\submanifold,g,f,\lambda)\), \(n>3\), be a complete gradient Yamabe
  soliton isometrically immersed into a space form \(\ambientspace{c}{n+1}\).
  Assume that \(f\) has exactly one critical point \(x_{0}\in{\submanifold}\).
  Then, either \(\submanifold\) is totally umbilical or else, locally, around
  every nonumbilical point not equal to \(x_{0}\), if
  \(\scalarcurvature\neq{n(n-1)c}\), it's rotationally symmetric.
\end{theorem}

\begin{proof}[Proof of Theorem~\ref{theoremx}]
  Notice that \((\submanifold,g)\) is a locally conformally flat manifold, by
  Theorem 1.1 of~\cite{catino2012global} (see also~\cite{cao2012structure}).

  Denote the set of umbilical points in \(\submanifold\) by \(\mathcal{U}\) and
  its relative complement by \(\mathcal{V}\). Then, because the second
  fundamental form is continuous, \(\mathcal{U}\) is closed and \(\mathcal{V}\)
  is open in \(\submanifold\). If \(\submanifold\) is not totally umbilical,
  that is, if \(\mathcal{V}\neq\emptyset\), then for every
  \(x\in{\mathcal{V}\setminus{\left\{x_{0}\right\}}}\) there exists some
  \(\delta_{x}>0\) such that \(B(x,\delta_{x})\subset\mathcal{V}\).

  By Cartan-Schouten's classification of locally conformally flat manifolds, we
  know that the second fundamental form of \(\submanifold\) has two eigenvalues
  on \(\submanifold\setminus\{x_{0}\}\) whose multiplicities are \(n-1\) and
  \(1\). Without loss of the generality, we assume of
  \begin{equation*}
    k_{1}
    =
    \cdots
    =
    k_{n-1}
    =
    \rho
    \quad\text{and}\quad
    k_{n}
    =
    \mu\neq\rho.
  \end{equation*}
  On the other hand,
  \[
    \scalarcurvature
    =
    n(n-1)c+n^{2}\meancurvature^{2}
    -
    \norm{\shapeoperator}^{2}
    =
    n(n-1)c
    +
    (n-1)(2\mu+(n-2)\rho)\rho.
  \]
  From the scalar curvature assumption we have that
  \begin{equation*}
    k_{1}=\cdots=k_{n-1}=\rho\neq0
    \quad\text{and}\quad
    k_{n}=\mu\neq\rho,
  \end{equation*}
  on \(B(x,\delta)\). Therefore, from Theorem 4.2 of do Carmo and
  Dajczer~\cite{do1983rotation}, \(B(x,\delta)\) is a rotational hypersurface
  of \(\ambientspace{c}{n+1}\).
\end{proof}

\begin{theorem}\label{finalt}
  Let \((\submanifold,g,f,\lambda)\) be a complete shrinking or steady gradient
  Yamabe soliton isometrically immersed into \(\euclideanspace{n+1}\) and
  assume that \(f\) has no critical points. If one of the following conditions
  hold:
  \begin{enumerate}[a)]
    \item
      \(n=2\);
    \item
      \begin{enumerate}[i.]
        \item
          \(n>2\),
        \item
          \(\sup\left\{\norm{\shapeoperator_{x}}:x\in{\submanifold}\right\}<\infty\),
          and
        \item
          \(\submanifold\) is locally conformally flat,
      \end{enumerate}
  \end{enumerate}
  then \((\submanifold,g,f,\lambda)\) is a trivial soliton. Moreover,
  \((\submanifold,g)\) must be one of the following submanifolds of euclidean
  space \(\euclideanspace{n+1}\): \(\mathbb{R}^n\) or
  \(\euclideanspace{1}\times\roundsphere{n-1}\).
\end{theorem}

\begin{proof}[Proof of Theorem~\ref{finalt}]
  Let \(n=2\). It follows from~\eqref{ww2} that
  \begin{equation*}
    \begin{split}
      0\leqslant\scalarcurvature&=-\frac{2f'''}{f'}.
    \end{split}
  \end{equation*}
  Hence \(f'\) is weakly concave on \(\reals\) and as such \(f'\) is
  constant. Therefore, we get from~\eqref{secc} that \(M^{2}\) is flat.
  By the Hartman and Nirenberg’s theorem, we have that \(M^{2}\) is either
  a plane or a cylinder.

  Now, let \(n\geqslant{3}\). From the work~\cite{cao2012structure} of Cao,
  Sun and Zhang, we know that \((\submanifold,g)\) takes, up to isometries,
  the structure of a warped product
  \[
    (\submanifold,g)
    =
    (\mathbb{R},dr^{2})\times_{\norm{\gradientfield{f}}}(N^{n-1},\bar{g}),
    \quad
    f'=\norm{\gradientfield{f}}>0.
  \]
  Therefore, we have the following expressions ruling its sectional curvature
  \begin{equation}\label{secc}
    \begin{split}
      R_{1a1b}
      &=
      -f'f''\bar{g}_{ab},
      \\
      R_{1abc}
      &=0,
      \\
      R_{abcd}
      &=
      (f')^{2}\overline{R}_{abcd}
      +
      (f'f'')^{2}(\bar{g}_{ad}\bar{g}_{bc}
      -
      \bar{g}_{ac}\bar{g}_{bd}),
    \end{split}
  \end{equation}
  and Ricci tensor
  \begin{equation}\label{secc2}
    \begin{split}
      \riccitensor_{11}
      &=
      -(n-1)\frac{f'''}{f'},
      \\
      \riccitensor_{1a}
      &=
      0,
      \\
      \riccitensor_{ab}
      &=
      \ambientriccitensor_{ab}
      -
      [(n-2)(f'')^{2}+f'f''']\bar{g}_{ab},
    \end{split}
  \end{equation}
  for every \(a,b,c,d\in{\left\{2,\ldots,n\right\}}\). Thus, its scalar
  curvature is given by
  \begin{equation}\label{ww2}
    \begin{split}
      \scalarcurvature
      &=
      \frac{1}{(f')^{2}}\overline{\scalarcurvature}
      -
      (n-1)(n-2)\left(\frac{f''}{f'}\right)^{2}-2(n-1)\frac{f'''}{f'},
    \end{split}
  \end{equation}
  and, from~\eqref{secc},~\eqref{secc2} and~\eqref{ww2}, we get
  \begin{equation}\label{secc3}
    \begin{split}
      W_{1a1b}
      &=
      \dfrac{\bar{\scalarcurvature}}{(n-1)(n-2)}\bar{g}_{ab}
      -
      \dfrac{1}{n-2}\bar{R}_{ab},
      \\
      W_{1abc}
      &=
      0,
      \\
      W_{abcd}
      &=
      (f')^{2}\overline{W}_{abcd},
    \end{split}
  \end{equation}
  for its Weyl tensor components. The Cotton tensor is given by
  \begin{align*}
    C_{1ab}
    &=
    \riccitensor_{ab,1}
    -
    \riccitensor_{1b,a}
    -
    \frac{1}{2(n-1)}(\scalarcurvature_{,1}g_{ab}-\scalarcurvature_{,a}g_{1b})
    \\
    &=
    \riccitensor_{ab,1}
    -
    \dfrac{1}{2(n-1)}\scalarcurvature_{,1}g_{ab}
    \\
    &=
    \frac{\scalarcurvature_{,1}}{2(n-1)}g_{ab}
    +
    \left(\frac{f'''}{f'}\right)_{1}g_{ab}.
  \end{align*}
  Since \(\submanifold\) is locally conformally flat, we get that \(C=0\). Then
  \begin{equation}\label{ww1}
    \frac{\scalarcurvature}{2(n-1)}
    +
    \frac{f'''}{f'}=Q(x),
    \quad
    x\in{N^{n-1}}.
  \end{equation}
  Because
  \begin{align*}
    0=C_{1a1}
    &=
    \riccitensor_{a1,1}
    -
    \riccitensor_{11,a}
    -
    \dfrac{1}{2(n-1)}(\scalarcurvature_{,1}g_{a1}
    -
    \scalarcurvature_{,a}g_{11})
    =
    -
    \frac{\scalarcurvature_{,a}}{2(n-1)}
  \end{align*}
  we conclude that \(Q=Q(x)\) is constant on \(\submanifold\). From~\eqref{ww2}
  and~\eqref{ww1} we get that
  \begin{equation}\label{font}
    \overline{\scalarcurvature}
    -
    (n-1)(n-2)(f'')^{2}
    =
    2(n-1)Q(f')^{2}.
  \end{equation}

  \begin{description}
    \item[Case 1:] \(Q=0\).

      Notice that \(f'\) is a weakly concave function on \(\reals\). Indeed,
      from~\eqref{font}, \(f''\) is constant and so \(f'''=0\) on \(\reals\).
      Thus, \(f'\) must be a weakly concave function on \(\reals\) as claimed.
      Therefore, \(f'\) is constant and the soliton
      \((\submanifold,g,f,\lambda)\) is trivial.

      It follows from~\eqref{ww1} and~\eqref{font} that
      \(\scalarcurvature=\bar{\scalarcurvature}=0\). Then, from~\eqref{secc} we
      have \(R_{1a1b}=0,R_{1abc}=0\) and
      \(R_{abcd}=(f')^{2}\overline{R}_{abcd}\). If \(n>3\), then we get
      from~\eqref{secc3} that \(\bar{N}\) is Einstein with \(\bar{W}=0\), i.e.,
      \(\bar{N}\) must be a space form. In this case \(\bar{R}_{abcd}=0\) and
      \(\submanifold\) is flat. If \(n=3\), then
      \[
        \bar{R}_{abcd}
        =
        \dfrac{\bar{\scalarcurvature}}{2}(\bar{g}_{ad}\bar{g}_{bc}
        -
        \bar{g}_{ac}\bar{g}_{bd}).
      \]
      on \(\bar{N}^{2}\). Since \(\bar{\scalarcurvature}=0\) we conclude that
      \(\submanifold\) is flat. In both cases we conclude that the sectional
      curvature of \(\submanifold\) vanishes identically. By Hartman and
      Nirenberg’s theorem~\cite{hartman1959spherical} we know that
      \(\submanifold\) is \(\mathbb{R}^n\).

    \item[Case 2:] \(Q>0\).

      In this case we have the following ODE:
      \begin{equation}\label{cccc}
        \overline{\scalarcurvature}-(n-1)(n-2)(f'')^{2}-2(n-1)Q(f')^{2}=0.
      \end{equation}
      Then either \(f'\) is constant or else
      \[
        f'(t)
        =
        c_{1}\cos\left(t\sqrt{\frac{2Q}{n-2}}\right)
        +
        c_{2}\sin\left(t\sqrt{\frac{2Q}{n-2}}\right),
      \]
      where \(c_{1},c_{2}\) are any real constants subject to
      \(\overline{\scalarcurvature}=2(n-1)Q(c_{1}^{2}+c_{2}^{2})\). Since
      \(f'>0\), we conclude that \(f'\) is a constant. In this case
      \(\scalarcurvature-\lambda=f''=0\), \(\scalarcurvature>0\),
      \(\bar{\scalarcurvature}>0\) and \(\bar{N}^{n-1}\) is a space form with
      positive curvature. It follows from~\eqref{secc} that \(\submanifold\)
      has nonnegative curvature. From a classification due to Cheng
      Yau~\cite{cheng1977hypersurfaces} we have that \(\submanifold\) is
      \(\mathbb{R}\times\roundsphere{n-1}\).

    \item[Case 3:] \(Q<0\).

      It follows from~\eqref{ww1} that
      \[
        \frac{\scalarcurvature}{2(n-1)}
        =
        Q-\frac{f'''}{f'}
        \geqslant
        0,
      \]
      that is,
      \[
        f'''\leqslant{Qf'}<0.
      \]
      Hence, \(f'\) is weakly concave on \(\reals\) and \(f'\) so is constant.
      This is a contradiction, since \(0\leqslant\scalarcurvature=2(n-1)Q<0\)
      by~\eqref{ww1}.

  \end{description}
\end{proof}

\bibliography{publish.bib}

\begin{thebibliography}{10}

\bibitem{albanese2013general}
Guglielmo Albanese, Luis~J Alias, and Marco Rigoli.
\newblock A general form of the weak maximum principle and some applications.
\newblock {\em Revista Matem{\'a}tica Iberoamericana}, 29(4):1437--1476, 2013.

\bibitem{alias2010scalar}
Luis~J Al{\'\i}as and S~Garc{\'\i}a-Mart{\'\i}nez.
\newblock On the scalar curvature of constant mean curvature hypersurfaces in
  space forms.
\newblock {\em Journal of mathematical analysis and applications},
  363(2):579--587, 2010.

\bibitem{alias2012estimate}
Luis~J Al{\'\i}as and S~Garc{\'\i}a-Mart{\'\i}nez.
\newblock An estimate for the scalar curvature of constant mean curvature
  hypersurfaces in space forms.
\newblock {\em Geometriae Dedicata}, 156(1):31--47, 2012.

\bibitem{alias2016maximum}
Luis~J Al{\'\i}as, Paolo Mastrolia, and Marco Rigoli.
\newblock {\em Maximum principles and geometric applications}, volume 700.
\newblock Springer, 2016.

\bibitem{cao2012structure}
Huai-Dong Cao, Xiaofeng Sun, and Yingying Zhang.
\newblock On the structure of gradient {Y}amabe solitons.
\newblock {\em Mathematical Research Letters}, 19(4):767--774, 2012.

\bibitem{catino2012global}
Giovanni Catino, Carlo Mantegazza, and Lorenzo Mazzieri.
\newblock On the global structure of conformal gradient solitons with
  nonnegative {R}icci tensor.
\newblock {\em Communications in Contemporary Mathematics}, 14(06):1250045,
  2012.

\bibitem{chen2018yamabe}
Bang~Yen Chen and Sharief Deshmukh.
\newblock Yamabe and quasi-{Y}amabe solitons on {E}uclidean submanifolds.
\newblock {\em Mediterranean Journal of Mathematics}, 15(5):194, 2018.

\bibitem{cheng1977hypersurfaces}
Shiu~Yuen Cheng and Shing~Tung Yau.
\newblock Hypersurfaces with constant scalar curvature.
\newblock {\em Mathematische Annalen}, 225:195--204, 1977.

\bibitem{chow1992yamabe}
Bennett Chow.
\newblock The {Y}amabe flow on locally conformally flat manifolds with positive
  {R}icci curvature.
\newblock {\em Communications on pure and applied mathematics},
  45(8):1003--1014, 1992.

\bibitem{cunha2021immersions}
Antonio~W Cunha and Eudes~L de~Lima.
\newblock Immersions of $r$-almost {Y}amabe solitons into {R}iemannian
  manifolds.
\newblock {\em manuscripta mathematica}, pages 1--12, 2021.

\bibitem{daskalopoulos2013classification}
Panagiota Daskalopoulos and Natasa Sesum.
\newblock The classification of locally conformally flat {Y}amabe solitons.
\newblock {\em Advances in Mathematics}, 240:346--369, 2013.

\bibitem{do1983rotation}
M~do~Carmo and M~Dajczer.
\newblock Rotation hypersurfaces in spaces of constant curvature.
\newblock {\em Transactions of the American Mathematical Society},
  277(2):685--709, 1983.

\bibitem{fujii2021classification}
Shunya Fujii and Shun Maeta.
\newblock Classification of generalized {Y}amabe solitons in {E}uclidean
  spaces.
\newblock {\em International Journal of Mathematics}, 32(04):2150022, 2021.

\bibitem{gilbarg2015elliptic}
David Gilbarg and Neil~S Trudinger.
\newblock {\em Elliptic Partial Differential Equations of Second Order}, volume
  224.
\newblock Springer Science \& Business Media, 2001.

\bibitem{hamilton1988ricciflow}
Richard~S Hamilton.
\newblock The {R}icci flow on surfaces, mathematics and general relativity
  (santa cruz, ca, 1986), 237--262.
\newblock {\em Contemp. Math}, 71:301--307, 1988.

\bibitem{hartman1959spherical}
Philip Hartman and Louis Nirenberg.
\newblock On spherical image maps whose {J}acobians do not change sign.
\newblock {\em American Journal of Mathematics}, 81(4):901--920, 1959.

\bibitem{hoffman1973surfaces}
David~A Hoffman.
\newblock Surfaces of constant mean curvature in manifolds of constant
  curvature.
\newblock {\em Journal of Differential Geometry}, 8(1):161--176, 1973.

\bibitem{hsu2012note}
Shu~Yu Hsu.
\newblock A note on compact gradient {Y}amabe solitons.
\newblock {\em Journal of Mathematical Analysis and Applications},
  388(2):725--726, 2012.

\bibitem{klotz1966complete}
Tilla Klotz and Robert Osserman.
\newblock Complete surfaces in $\mathbb{E}^3$ with constant mean curvature.
\newblock {\em Commentarii Mathematici Helvetici}, 41(1):313--318, 1966.

\bibitem{ma2012remarks}
Li~Ma and Vicente Miquel.
\newblock Remarks on scalar curvature of {Y}amabe solitons.
\newblock {\em Annals of Global Analysis and Geometry}, 42(2):195--205, 2012.

\bibitem{mastrolia2013some}
Paolo Mastrolia, Marco Rigoli, and Michele Rimoldi.
\newblock Some geometric analysis on generic {R}icci solitons.
\newblock {\em Communications in Contemporary Mathematics}, 15(03):1250058,
  2013.

\bibitem{nomizu1969formula}
Katsumi Nomizu and Brian Smyth.
\newblock A formula of {S}imons' type and hypersurfaces with constant mean
  curvature.
\newblock {\em Journal of Differential Geometry}, 3(3-4):367--377, 1969.

\bibitem{pigola2011remarks}
S.~Pigola, M.~Rimoldi, and A.~G. Setti.
\newblock Remarks on non-compact gradient {R}icci solitons.
\newblock {\em Mathematische Zeitschrift}, 268(3-4):777--790, 2011.

\bibitem{pigola2005maximum}
Stefano Pigola, Marco Rigoli, and Alberto~Giulio Setti.
\newblock {\em Maximum principles on {R}iemannian manifolds and applications}.
\newblock American Mathematical Soc., 2005.

\bibitem{seko2019classification}
Tatsuya Seko and Shun Maeta.
\newblock Classification of almost {Y}amabe solitons in {E}uclidean spaces.
\newblock {\em Journal of Geometry and Physics}, 136:97--103, 2019.

\bibitem{suh2023yamabe}
Young~Jin Suh.
\newblock Yamabe and quasi-{Y}amabe solitons on hypersurfaces in the complex
  hyperbolic space.
\newblock {\em Mediterranean Journal of Mathematics}, 20(2):69, 2023.

\bibitem{tokura2020immersion}
W~Tokura, L~Adriano, E~Batista, and AC~Bezerra.
\newblock Immersion of gradient almost yamabe solitons into warped product
  manifolds.
\newblock {\em arXiv preprint arXiv:2010.03995}, 2020.

\bibitem{tribuzy1978hopf}
Renato Tribuzy.
\newblock Hopf method and deformations of surfaces preserving mean curvature.
\newblock {\em Anais da Academia Brasileira de Ciências}, 50(4):446--450,
  1978.

\end{thebibliography}
\bibliographystyle{plain}

\end{document}